\documentclass[10pt]{article}
\usepackage{amsmath}
\usepackage{amsfonts, amssymb, amsthm}
\usepackage{epsfig}
\usepackage{color}
\usepackage{multicol}

\newlength{\hchng}
\newlength{\vchng}
\newtheorem{thm}{Theorem}[section]
\newtheorem{prop}[thm]{Proposition}
\newtheorem{cor}[thm]{Corollary}

\newtheorem{lemma}[thm]{Lemma}

\newtheorem{preremark}[thm]{Remark}
\newenvironment{remark}{\begin{preremark}\rm}{\medskip \end{preremark}}
\numberwithin{equation}{section}

\newcommand{\norm}[1]{\left\Vert#1\right\Vert}
\newcommand{\abs}[1]{\left\vert#1\right\vert}

\newcommand{\R}{\mathbb R}
\newcommand{\eps}{\varepsilon}

\newcommand{\grad} {\nabla}
\newcommand{\lap} {\triangle}

\newcommand{\dd} {\; \mathrm{d}}

\DeclareMathOperator{\dv}{div}

\title{On the differentiability of the solution to an equation with drift and fractional diffusion}
\author{Luis Silvestre}

\begin{document}
\maketitle

\begin{abstract}
We consider an equation with drift and either critical or supercritical fractional diffusion. Under a regularity assumption for the vector field that is marginally stronger than what is required for H\"older continuity of the solutions, we prove that the solution becomes immediately differentiable with H\"older continuous derivatives. Therefore, the solutions to the equation are classical.
\end{abstract}

\section{Introduction}

The purpose of this paper is to study the regularity assumptions needed on a vector field $b$ for the solutions to an equation with drift and fractional diffusion
\begin{equation} \label{e:dd}
u_t + b \cdot \grad u + (-\lap)^s u = f
\end{equation}
to be differentiable (and therefore classical). We do \textbf{not} assume the vector field $b$ to be divergence free.

We will concentrate in the so called supercritical case $s \in (0,1/2]$. The case $s=1$ is usually referred to as the critical case, and $s \in (0,1/2)$ as the supercritical regime because the drift term is of higher order than the diffusion term. In \cite{silvestrePreprintHolder} it is shown that if $b \in C^{1-2s}$ in space for $s<1/2$ or $b \in L^\infty$ for $s=1/2$, then $u$ becomes H\"older continuous for positive time. In this paper we prove that with the slightly better regularity assumption $b \in C^{1-2s+\alpha}$ for any $\alpha \in (0,2s)$, we obtain that the solution $u$ becomes $C^{1,\alpha}$ in space. This is large jump in regularity for a seemingly minimal extra regularity assumption in $b$. Note that $b \in C^{1-2s}$ is the assumption that matches the scaling of the equation. A slightly better assumption like $b \in C^{1-2s+\alpha}$ allows us to use local perturbative techniques and thus obtain much better regularity results.

For the case $s \geq 1/2$, a related result has recently been obtained in \cite{priola2010pathwise}, where a similar regularity for an eigenvalue problem is established.

Our main result is the following.

\begin{thm} \label{t:main}
Let $b$ be a vector field in $L^\infty([-1,0],C^{1-2s+\alpha}(B_1))$ and $f \in L^\infty([-1,0],C^{1-2s+\alpha}(B_1))$ for some $\alpha \in (0,2s)$, then any bounded solution $u$ of \eqref{e:dd} in $[-1,0] \times B_1$ is $C^{1,\alpha}$ in space. Moreover, an estimate holds
\[ ||u||_{L^\infty([-1/2,0],C^{1,\alpha}(B_{1/2}))} \leq C \left( ||u||_{L^\infty([-1,0] \times \R^n)} + ||f||_{L^\infty([-1,0],C^{1-2s+\alpha}(B_1))} \right) \]
where the constant $C$ depends on $s$, $n$ and $||b||_{C^{1-2s+\alpha}}$ only.
\end{thm}

If $b \in C^{1-2s}$ and $f \in C^{1-2s}$ (case $\alpha=0$) we can prove that $u$ is almost Lipschitz in space making a mild local smallness assumption. We state that as a second theorem. The proof is contained in the proof of Theorem \ref{t:main}.
\begin{thm} \label{t:main2}
There exists a $\delta>0$ such that the following result holds. Let $b \in [-1,0] \times B_1$ be a vector field such that for some $r >0$ and any $t \in [-1,0]$,
\begin{equation} \label{e:mainassumption2}
 \sup_{|x-y|<r} \frac{|b(t,x)-b(t,y)|}{|x-y|^{1-2s}} \leq \delta.
\end{equation}
Assume also that $f \in C^{1-2s}$ in space, then any bounded solution $u$ of \eqref{e:dd} in $[-1,0] \times B_1$ is $C^\beta$ in space for all $\beta < 1$. Moreover, an estimate holds
\[ ||u||_{L^\infty([-1/2,0],C^\beta(B_{1/2}))} \leq C \left( ||u||_{L^\infty([-1,0] \times \R^n)} + ||f||_{L^\infty([-1,0],C^{1-2s}(B_1))} \right) \]
where the constant $C$ depends on $r$, $s$, $n$, $\beta$ and $||b||_{C^{1-2s}}$ only.
\end{thm}

The focus of this paper is on the regularity estimates for the solution $u$ and not on the existence or uniqueness for a given vector field $b$ and initial value $u(x,0) = u_0$. It is our intention to use this result as a useful criteria to conclude that a solution is classical that could be applied to a variety of nonlinear equations with drift and fractional diffusion. Some applications are:
\begin{itemize}
\item It implies the result in \cite{constantin2008regularity} about a regularity criteria for the supercritical surface quasi-geostrophic equation.
\item It complements \cite{chan2009eventual} and \cite{kiselev2011nonlocal} to prove that the solutions to the supercritical Burgers equation become classical for large time, instead of only H\"older continuous.
\item Combined with \cite{silvestrePreprintHolder}, it implies the result in \cite{gautam2008}. Moreover, it proves a more general version of the result where $u = L(-\lap)^{-s}\theta$ for any operator $L$ of order zero (any combination of Riesz transforms or singular integrals), without assuming that $\dv u=0$.
\item For any conservation law with fractional diffusion
\[ \theta_t + \dv F(\theta) + (-\lap)^s \theta = 0 \]
with $F$ being a smooth nonlinear function, it proves that if $\theta \in C^{1-2s+\alpha}$, then actually $\theta \in C^{1,\alpha}$ and it is a classical solution.
\item For the Hamilton Jacobi equation with fractional diffusion
\[ u_t + H(\grad u) + (-\lap)^s u = 0, \]
if $H$ is smooth, it implies that if $u \in C^{1,1-2s+\alpha}$ for some $\alpha>0$ then, applying Theorem \ref{t:main} to the directional derivatives of $u$,  $u \in C^{2,\alpha}$ for all $\alpha < 2s$. In the case of $s=1/2$, this provides an improvement of the regularity obtained in \cite{silvestre2009differentiability}.
\end{itemize}

We also plan to use this result in future work for obtaining partial regularity results for nonlinear problems for which some decay can be proved. For example, we expect that the methods in this paper would allow us to improve the results from \cite{silvestre2009eventual} and \cite{chan2009eventual}, although these two results have recently been improved with different methods in \cite{dabkowskieventual} and \cite{kiselev2011nonlocal}.

We stress that there is no assumption on the divergence of $b$ for either Theorem \ref{t:main} or Theorem \ref{t:main2}. This allows us to apply the result to a larger class of equations. Moreover, there is no energy inequality for equation \eqref{e:dd} and variational methods do not seem to be suitable for a proof.

The idea of the proof is to approximate locally the solution to the drift-diffusion equation by a solution to a drift-less problem via a change of variables (which is equivalent to approximating the solution locally by the solution to a problem with constant drift). The drift-less problem has $C^\infty$ solutions, so we use it to show that there is a plane from which the solution separates slowly. A precise estimate on this separation leads to the $C^{1,\alpha}$ estimate. This general idea goes back to the classical result of Cordes \cite{cordes1956} and Nirenberg \cite{nirenberg1954} for second order elliptic equations. For stationary integro-differential equation of order larger than one, it was developed in great generality in \cite{caffarelli2009approx}. There is a difficulty that arises when trying to apply this method to integral equations of order one or less. The problem is that the affine functions that are used in the approximations have a linear growth at infinity that already makes the tails of the integrals divergent when $s \leq 1/2$. In this paper this difficulty is overcome by rewriting the equation using the extension from \cite{caffarelli2007extension} and complementing the affine function in the original variable $x$ with an extra term in the variable $y$ with appropriate homogeneity (depending on $s$).

The organization of the paper is as follows. In section \ref{s:extension} we explain the extended problem. The value $1+2s$ appears as the order of the second term in the expansion of solutions to the extended problem, and that is the main reason of the constraint $\alpha < 2s$ in Theorem \ref{t:main} even for $f=0$. In section \ref{s:weaksolutions} we discuss the applicability of Theorem \ref{t:main} to viscosity solutions. Essentially, the proof uses only the comparison principle with smooth sub and supersolutions, and thus it adapts perfectly to the case of viscosity solutions. In sections \ref{s:interiorEstimatesForFHE} and \ref{s:boundaryEstimatesForFHE}, we develop the necessary regularity estimates for the drift-less problem. These estimates are not difficult and they take a large proportion of this article. But they need to be proved because the fractional heat equation is not classical in its extended version. The proofs in section \ref{s:interiorEstimatesForFHE} and \ref{s:boundaryEstimatesForFHE} use essentially the same ideas as the usual heat equation and hence they are very standard. On a first read of this paper, it would be convenient to just skim through the results in these two sections. In section \ref{s:holderEstimates}, the H\"older estimates from \cite{silvestrePreprintHolder} are adapted to the extended problem. Sections \ref{s:perturbation}, \ref{s:improvementOfFlatness} and \ref{s:mainProof} are where the most important work of this paper is carried out specifically in order to prove Theorem \ref{t:main}.

\noindent \textbf{Notation: }
\begin{itemize}
\item The variable $x$ will always be used for a point in $\R^n$ and $y$ to be the extra coordinate in $\R^+$ for the extension to the upper half space. Sometimes the capital letter $X$ will be used to denote a point in the upper half space $\R^n \times \R^+$. Thus $X = (x,y)$.
\item By $B_1^+$ we denote the half ball centered at $(0,0)$ in the upper half space. i.e. $B_1^+ = \{ (x,y) \in \R^n \times \R^+ : |x|^2 + y^2 < 1 \}$.
\item By $B_1^0$ we denote the unit ball either in the original space $\R^n$ or in the boundary of the upper half space $\R^n \times \{0\}$.
\item By $(\partial B_1)^+$ we denote the upper half part of the boundary of the unit ball $\partial B_1$ for which $y>0$. Thus $(\partial B_1)^+ = \partial B_1 \cap \{y>0\}$. Also $\partial B_1^+ = B_1^0 \cup (\partial B_1)^+$.
\end{itemize}

\section{The extended problem}
\label{s:extension}
In \cite{caffarelli2007extension}, the fractional Laplacian was characterized as a Dirichlet to Neumann problem for a degenerate elliptic equation. More precisely, given a function $f \in \R^n \to \R$, we can compute $(-\lap)^s f$ by solving the following Dirichlet problem in the upper half space
\begin{equation} \label{e:extension}
\begin{aligned}
u(x,0) &= f(x) \ \ \text{in } \R^n \\
\dv y^a \grad u &= 0 \ \ \text{in } \R^n \times \R^+
\end{aligned}
\end{equation}
Then $(-\lap)^s f = c\lim_{y \to 0} y^a \partial_y u$. Here $\grad u$ stands for the full gradient respect to the original variable $x \in R^n$ and the extended variable $y \in \R^+$. The constant $a$ is equal to $1-2s$. The most effective way to remember the relation between $a$ and $s$ is to use dimensional analysis. The order of the operator $(-\lap)^s$ is $2s$, whereas $y^a \partial_y u$ is of order $1-a$ ($-a$ because of $y^a$ and $+1$ because of $\partial_y$), thus $2s=1-a$.

Note that since in this paper we focus on the case $s \in (0,1/2)$, the value of $a$ stays in the range $a \in (0,1)$.

We rewrite the equation \eqref{e:dd} using the characterization of the fractional Laplacian as a Dirichlet to Neumann operator given in \cite{caffarelli2007extension}.

\begin{equation} \label{e:dd-extended}
\begin{aligned}
u_t(t,x,0) + b(t,x) \cdot \grad_x u - c \lim_{y \to 0} y^a \partial_y u(t,x,y) &= f(t,x) && \text{for } x \in \R^n \text{ and } t \in (0,\infty) \\
\dv y^a \grad u &= 0 && \text{for } y > 0
\end{aligned}
\end{equation}

The multiplicative constant $c$ depends on the value of $s$ (or, equivalently, on $a$). In terms of the smoothness of the solution of \eqref{e:dd-extended}, the value of $c$ is irrelevant, since it can be modified with a simple change of variables. So we will study the equation \eqref{e:dd-extended} with $c=1$.

The usefulness of this construction comes from the fact that the equation has become local. There is no integro-differential operator in \eqref{e:dd-extended}, so the usual difficulty of keeping track of errors coming from the tails of the integrals that one has to deal with when studying regularity of integro-differential equations disappears in \eqref{e:dd-extended}.

In the next lemmas, we will explore the behavior close to $y=0$ of the solution to \eqref{e:extension} when the original function $f$ is smooth.

It is useful to remember some special harmonic functions for the extension \eqref{e:extension}. The following list plays the role of the first few harmonic polynomials.

\begin{itemize}
\item For any constant vector $A \in \R^n$, the function $u(x,y)= A \cdot x$ solves \eqref{e:extension} and $\lim_{y \to 0} y^a \partial_y u = 0$. 
\item The function $u(x,y) = \frac 1 {1-a} y^{1-a}$ solves \eqref{e:extension} and $\lim_{y \to 0} y^a \partial_y u = 1$. 
\item For any constant vector $A \in \R^n$, the function $u(x,y) = \frac 1 {1-a} y^{1-a} \ A \cdot x$ solves \eqref{e:extension} and $\lim_{y \to 0} y^a \partial_y u = A \cdot x$. 
\item The function $u(x,y)= x^2 - \frac{n}{1+a} y^2$ solves \eqref{e:extension} and $\lim_{y \to 0} y^a \partial_y u = 0$. 
\end{itemize}

Functions of the form $A \cdot x + \frac B {1-a} y^{1-a}$ play the role of the linear harmonic functions. Note that the next \emph{harmonic polynomial} would be of the form $A \cdot x \ \frac 1 {1-a} y^{1-a}$, which has degree $2-a = 1+2s$. The following lemmas show that a first order approximation of smooth functions $f$ close to the boundary in the extension problem \eqref{e:extended} has an error of order $1+2s$. This is the reason why in our main theorem, even if the right hand side is zero, we do not obtain a regularity estimate of the solution $u$ in the space $C^{1,2s}$ or better.

\begin{lemma} \label{l:ext-fglobal}
 Let $f: \R^n \to \R$ be a bounded function, $f \in C^{1,1}(B_1)$, and $u$ be its extension to the upper half space:
\begin{align*}
u(x,0) &= f(x) \ \text{on } \R^n \times \{0\} \\
\dv y^a \grad u&= 0 \qquad \text{in } \R^n \times (0,+\infty) 
\end{align*}
Then, for $x \in B_{1/2}$, $u$ has the following expansion
\begin{equation} \label{e:ext-global-expansion}
u(x,y) = f(x) + y^{1-a} g(x) + O(y^2).
\end{equation}
More precisely, the error in the expansion is bounded by $C y^2 (||f||_{L^\infty(\R^n)} + ||f||_{C^{1,1}(B_1)})$ for a constant $C$ depending only on $a$ and dimension.
\end{lemma}

\begin{proof}
Note that since $f$ is bounded, $u$ is bounded as well from the maximum principle. Therefore, the estimate of the expansion is trivially true for $y>1$. We need to check the case of $y$ smaller than $1$.

We compute $u$ in terms of $f$ explicitly using the Poisson kernel from \cite{caffarelli2007extension}.
\[ u(x,y) = C_s \int_{\R^n} f(x+z) \frac{y^{1-a}}{(|z|^2+y^2)^{\frac{n+1-a}{2}}} \dd z \]
We consider $g(x) = -c (-\lap)^s f$ and we have
\begin{align*}
f(x) &= C_s \int f(x) \frac{y^{1-a}}{(|z|^2+y^2)^{\frac{n+1-a}{2}}} \dd z \\
y^{1-a} g(x) &= C_s \int (f(x+z) - f(x)) \frac{y^{1-a}}{|z|^{n+1-a}} \dd z \\
\end{align*}
Therefore, we estimate the error in \eqref{e:ext-global-expansion}.
\begin{align*}
|u(x,y) - & f(x) - y^{1-a} g(x)|= \\
&= C_s \abs{\int (f(x+z) - f(x)) y^{1-a} \left( \frac 1 {(|z|^2+y^2)^{\frac{n+1-a}{2}}} - \frac 1 {|z|^{n+1-a}} \right) \dd z} \\
\intertext{Since the kernel is even, the first order part of $f$ in $B_{1/2}$ intergrates to zero and we can estimate the integral with the $C^{1,1}$ norm of $f$.}
&\leq C \int (|z|^2 \chi_{B_{1/2}} + \chi_{\R^n \setminus B_{1/2}} ) y^{1-a} \abs{ \frac {|z|^{n+1-a} - (|z|^2+y^2)^{\frac{n+1-a}{2}}} {(|z|^2+y^2)^{\frac{n+1-a}{2}} |z|^{n+1-a}} } \dd z \\
\end{align*}
In order to estimate the value of this integral, we split the domain of integration in two subdomains: $\R^n \setminus B_{y/2}$ and $B_{y/2}$.

For $|z| > y/2$, we can estimate the numerator in the kernel as
\[ \abs{|z|^{n+1-a} - (|z|^2+y^2)^{\frac{n+1-a}{2}}} \leq C |z|^{n-1-a} y^2. \]

For $z \in \R^n \setminus B_{y/2}$, we estimate the integral
\begin{align*}
C \int_{\R^n \setminus B_{y/2}} |z|^2 y^{1-a} \frac{C |z|^{n-1-a} y^2 } {(|z|^2+y^2)^{\frac{n+1-a}{2}} |z|^{n+1-a}}  \dd z &= C y^2 \int_{\R^n \setminus B_{y/2}} \frac {y^{1-a}} {(|z|^2+y^2)^{\frac{n+1-a}{2}}} \dd z \\
& = C y^2 \int_{\R^n \setminus B_{y/2}} \frac {y^{-n}} {(|z/y|^2+1)^{\frac{n+1-a}{2}}} \dd z = C y^2
\end{align*}

Finally, for $|z|< y/2$, we have
\begin{align*}
(|z|^2+y^2)^{\frac{n+1-a}{2}} &\approx  C y^{n+1-a} \\
\abs{|z|^{n+1-a} - (|z|^2+y^2)^{\frac{n+1-a}{2}}} &\approx C y^{n+1-a}
\end{align*}
Thus, we estimate
\begin{align*}
C \int_{B_{y/2}} |z|^2 y^{1-a} \abs{ \frac {|z|^{n+1-a} - (|z|^2+y^2)^{\frac{n+1-a}{2}}} {(|z|^2+y^2)^{\frac{n+1-a}{2}} |z|^{n+1-a}} } \dd z \leq C \int_{B_{y/2}}  \frac{ |z|^2 y^{1-a} } {|z|^{n+1-a}}  \dd z = Cy^2
\end{align*}
Which finishes the estimate of the error in \eqref{e:ext-global-expansion}.
\end{proof}

The following lemma is a localized version of Lemma \ref{l:ext-fglobal}.

\begin{lemma} \label{l:ext-flocal}
Let $u$ solve
\begin{align*}
u(x,0) &= f(x) \ \text{on } B_1^0 \\
\dv y^a \grad u &= 0 \qquad \text{in } B_1^+
\end{align*}
Assume that $f \in C^{1,1}(B_1^0)$ and $u \in L^2(B_1^+,y^a)$. Then $u$ has the following expansion in $B_{1/2}^+$.
\begin{equation} \label{e:ext-expansion-local}
 u(x,y) = f(x) + y^{1-a} g(x) + O(y^2)
\end{equation}
for a function $g \in C^{1+a}$ and the error is controlled by $Cy^2(||f||_{C^{1,1}} + ||u||_{L^2(y^a)})$, where $C$ depends on $n$ and $a$ only.
\end{lemma}

\begin{proof}
Let $v$ be the extension to the upper half space of the function $f$ extended to $\R^n$ as zero outside $B_1^0$.
\begin{align*}
v(x,0) &= \begin{cases}
f(x) & \text{for } |x|<1 \\
0 & \text{for } |x| \geq 1
\end{cases} \\
\dv y^a \grad v &= 0 \qquad \text{in } \R^n \times (0,+\infty)
\end{align*}
We apply Lemma \ref{l:ext-fglobal} and obtain an expansion like \eqref{e:ext-expansion-local} for $v$.

We are left to show that the remainder $w = u-v$ has an expansion of the same kind. From the maximum principle, $|v|$ is bounded by $||f||_{L^\infty}$, and therefore $||v||_{L^2(B_1^+,y^a)} \leq C ||f||_{L^\infty}$ and $||w||_{L^2(B_1^+,y^a)} \leq C (||f||_{L^\infty} + ||u||_{L^2(B_1^+,y^a)})$.

We can consider the odd extension of $w$ to the full ball $B_1$ by $w(x,-y) = -w(x,y)$ so that $w$ solves the degenerate elliptic equation across $\{y=0\}$:
\[ \dv |y|^a \grad w = 0 \ \text{in } B_1. \]
We apply the Harnack inequality from \cite{FKS} to $w$. Then $w$ is H\"older continuous in the interior of $B_1$. Moreover, from the Cacciopoli inequality
\begin{equation} \label{e:cacciopoli}
\int_{B_{7/8}} y^a |\grad w|^2 \leq C \int_{B_1} y^a |w|^2.
\end{equation}

We will apply the boundary Harnack principle from \cite{FKJ} to $w$ in $B_1^+$. In order to apply that theorem, we split $w$ into its positive and negative parts
\begin{multicols}{2}
\begin{align*}
w^p &= w^+  \ \text{on } \partial B_{3/4}^+ \\
\dv y^a \grad w^p &= 0 \qquad \text{in } B_{3/4}^+
\end{align*}

\begin{align*}
w^n &= w^-  \ \text{on } \partial B_{3/4}^+ \\
\dv y^a \grad w^n &= 0 \qquad \text{in } B_{3/4}^+
\end{align*}
\end{multicols}
Thus, $w = w^p - w^n$, $w^p = w^n = 0$ on $B_{3/4}^0$ and both $w^p$ and $w^n$ are nonnegative. Since $\dv y^a \grad (y^{1-a}) = 0$ in the upper half space $\{y > 0\}$ we can apply the Boundary Harnack theorem from \cite{FKJ} to obtain
that both $w^p / y^{1-a}$ and $w^n/y^{1-a}$ are H\"older continuous up the the boundary in $\overline{B_{5/8}^+}$. Therefore, there exist a H\"older continuous function $\tilde g$ such that
\[ w(x,y) = \tilde g(x,y) y^{1-a} \qquad \text{in } \overline{B_{5/8}^+}. \]

Recalling that $u = v+w$, so far we have proved that $u$ has an expansion of the form $u(x,y) = f(x) + y^{1-a} g(x) + O(y^{1-a+\alpha})$ for some $\alpha>0$ (the H\"older exponent in $\tilde g$).

We can repeat the same argument for any derivative (or incremental quotient) of $w$ with respect to $x$. From Cacciopoli's inequality, all those derivatives are in $L^2(y^a)$ and also solve the same equation. Therefore we obtain
\[ w_{x_i}(x,y) = \tilde g_{x_i}(x,y) y^{1-a} \qquad \text{in } \overline{B_{5/8}^+}. \]
with $g_{x_i}$ H\"older continuous. Repeating this argument, we obtain that $\tilde g$ is $C^\infty$ in the $x$ variable.

Now consider $u_2(x,y) = y^a \partial_y w$. Then we have
\begin{align*}
\int y^{-a} u_2^2 &= \int_{B_{5/8}^+} y^a |\partial_y w|^2 \leq C \int_{B_{3/4}^+} y^a |w|^2 \leq C (\norm{f}_{L^\infty} + ||u||_{L^2(B_1^+,y^a)})^2 \\
u_2(x,0) &= \tilde g(x,0)
\end{align*}
Moreover, as it is pointed out in \cite{caffarelli2007extension}, $u_2$ satisfies the conjugate equation
\[ \dv y^{-a} \grad u_2 = 0 \text{ in } B_{5/8}^+ \]
So we can start over out argument with $u_2$ instead of $u$ and $-a$ instead of $a$ to obtain a H\"older continuous function $h$ such that
\[ u_2(x,y) = \tilde g(x,0) + y^{1+a} h(x,y) \qquad \text{in } \overline{B_{1/2}^+}. \]
Therefore \[ w(x,y) = \int_0^y \partial_y w(x,z) \dd z = \int_0^y \partial_y y^{-a} \tilde g(x,0) + y h(x,y) \dd z = y^{1-a} \tilde g(x,0) + O(y^2) \]
Which finishes the proof.
\end{proof}

\begin{remark}
Note that the function $g(x)$ corresponds to the fractional Laplacian of $f$.
\[ g(x) = \lim_{y \to 0} \frac {u(x,y)}{y^{1-a}} = \frac 1 {1-a} \lim_{y \to 0} y^a \partial_y u(x,y) \]
\end{remark}

\begin{cor}
Let $f : B_1^0 \to \R$ be a $C^{1,1}$ function and $u: B_1^+ \to \R$ satisfy the extension PDE:
\begin{align*}
u(x,0) &= f \ \ \text{on } B_1^0 \\
\dv (y^a \grad u) &= 0 \ \ \text{in } B_1^+
\end{align*}
where $a>0$. Then there exists a number $D$ and $C>0$ such that
\begin{equation} \label{e:ext-normalderivative}
|u(x,y) - f(0) - \grad f(0) \cdot x - D y^{1-a}| \leq C (x^2+y^2+|x|y^{1-a})
\end{equation}
where $C$ and an upper bound for $|D|$ depend on $||f||_{C^2(B_1^0)}$ and $||u||_{L^2(B_1^+,y^a)}$.
\end{cor}

\begin{proof}
From Lemma \ref{l:ext-flocal}, $u$ has an expansion of the form
\[ u(x,y) = f(x) + y^{1-a} g(x) + O(y^2) \]
Since $f \in C^{1,1}$, $f(x) = f(0) + x \cdot \grad f(0) + O(|x|^2)$.

The function $g$ is $C^{1+a}$. Since $a>0$, then $g$ is Lipschitz. If we choose $D = g(0)$, we have
\[ g(x) y^{1-a} = g(0) y^{1-a} + O(|x| y^{1-a}). \]

Adding the three terms together, we finish the proof.
\end{proof}


\section{Weak solutions}
\label{s:weaksolutions}

\subsection{Viscosity solutions}
Since we do not make any assumption on the divergence of $b$, integration by parts is complicated and the distributional sense is not suitable for defining weak solutions of \eqref{e:dd}. Since we are assuming that $b$ is continuous, we can use the concept of viscosity solutions developed originally by Crandall and Lions for Hamilton-Jacobi equations. The most straight forward way to define it for the integral equation \eqref{e:dd} is the following.

We say that a continuous function $u$ is a viscosity solution of \eqref{e:dd} if for every point $(x_0,t_0) \in (-1,0] \times B_1^0$ and every bounded function $\phi : [-1,t_0] \times \R^n$ such that
\begin{enumerate}
\item $\varphi$ is smooth around $(t_0,x_0)$,
\item $\phi(t_0,x_0)=u(t_0,x_0)$,
\item $\phi \geq u$ (or respectively $\phi \leq u$) in $[-1,t_0] \times \R^n$,
\end{enumerate}
then
\[ \phi_t(t_0,x_0) + b(t_0,x_0) \cdot \grad \phi(t_0,x_0) + (-\lap)^s \phi(t_0,x_0) \leq f(t_0,x_0) \ (\text{ or } \geq f(t_0,x_0) \text{ respectively.}) \]

This is the definition used in \cite{silvestre2009differentiability}. in the same way as distributional solutions of PDEs are defined based on the property of integration by parts against smooth functions, viscosity solutions are defined based on the property of the comparison principle with smooth sub and super solutions. This concept is more apparent in the following definition of viscosity solution for the extended problem \eqref{e:dd-extended}. Indeed, the definition says that a continuous function $u$ is a viscosity solution if it satisfies the appropriate comparison conditions with respect to smooth sub and super solutions. We give the precise definition below.

We say that a continuous function $u : [-1,0] \times \overline B_1^+ \to \R$ solves \eqref{e:dd-extended} in the viscosity sense if for any $t \in [-1,0]$, $u(t,x,y)$ solves the Dirichlet problem \eqref{e:extension} (classically) in $\{t\} \times B_1^+$ and for every function $\phi : [t_0-\eps,t_0] \times B^+_\eps(x_0) \to \R$ such that
\begin{enumerate}
\item $\varphi$ is smooth in $[t_0-\eps,t_0] \times B_\eps^0(x_0)$,
\item $\varphi$ satisfies $\dv y^a \grad \phi=0$ in $B_\eps^+(x_0)$ for all $t \in [t_0-\eps,t_0]$,
\item $\phi_t + b \cdot \grad \phi - \lim_{y \to 0} y^a \partial_y \phi \geq f \text{ (or } \leq f \text{ respectively) in } (t_0-\eps,t_0]\times B^0_\eps(x_0)$
\item $\phi \geq u \text{ (or } \leq u \text{ respectively) on } \{t_0-\eps\} \times B^+_\eps(x_0) \cup [t_0-\eps,t_0] \times (\partial B_\eps(x_0))^+,$
\end{enumerate}
then $\phi \geq u$ or $\phi \leq u$ respectively, in the whole domain $[t_0-\eps,t_0] \times B_\eps^+(x_0)$.

It is not hard to prove that the definitions of viscosity solution for \eqref{e:dd} implies the corresponding definition for \eqref{e:dd-extended} after the extension \eqref{e:extension}.

Viscosity solutions is the most appropriate type of weak solutions for the results in this paper. The actual equation is used directly only in the proof of Lemma \ref{l:perturbation}, and the way it is used is to compare the solution with a classical subsolution.

Note that for any other type of weak solution $u$ (for example entropy solutions), the same function would solve the equation in the viscosity sense as long as the comparison principle with classical sub and super-solutions holds. In practical cases, a drift-fractional diffusion equation may be obtained as the linearization of an equation for which the comparison principle can be shown to apply, and thus the results of this paper would apply as well.

\subsection{The vanishing viscosity method}
One way to avoid the concept of weak solutions completely is by adding a vanishing viscosity term and proving uniform estimates. This method is usually called the \emph{vanishing viscosity method}. In this context, it is not related to the concept of viscosity solutions at all despite the unfortunate similarity of the names.

The idea consists in adding an artificial viscosity term to the equation
\[ u_t + b \cdot \grad u + (-\lap)^s u - \eps \lap u = f \ \ \text{in } [-1,0] \times \R^n. \]
With the added artificial viscosity $\eps \lap u$, the equation has a smooth classical solution.

If we prove that the estimates in this paper hold uniformly as $\eps \to 0$, then the existence of classical solutions for \eqref{e:dd} follows by passage to the limit. In terms of the extended equation, the same idea applies to
\begin{align*}
u_t + b \cdot \grad u -(\lim_{y \to 0} y^a \partial_y u) - \eps \lap u &= f && \text{in } \ \ (-1,0] \times B_1^0.\\
\dv y^a \grad u &= 0 && \text{in } \ \ [-1,0] \times B_1^+.
\end{align*}

The estimates in this paper are indeed uniform as $\eps \to 0$. The explicit form of the equation \eqref{e:dd-extended} is used only for Proposition \ref{p:holder-local-extended} and Lemma \ref{l:perturbation}. The first is a localized version of the result in \cite{silvestrePreprintHolder}, which adapts to the vanishing viscosity method (as pointed out in that paper). The second also adapts to the vanishing viscosity method as it is analyzed in Remark \ref{r:perturbation}.

The vanishing viscosity method is a very common approach for nonlinear conservation laws. It is used for example in \cite{caffarelli2006drift}, \cite{kiselev2010variation} and \cite{MR2259335} for problems with drift and fractional diffusion.

In this paper we prefer to focus on the estimates but we will point out precisely how the proofs adapt to either continuous viscosity solutions or to the vanishing viscosity approximation. It turns out that the precise form of the equation is used only in very specific steps in the paper, so it is only necessary to analyze the type of solution at those points (See remarks \ref{r:perturbation0} and \ref{r:perturbation})

\section{Interior estimates for the fractional heat equation}
\label{s:interiorEstimatesForFHE}
In order to carry out the perturbative arguments to prove the main theorem, we need to have good local regularity estimates for the drift-less equation.
\begin{equation} \label{e:fh-extended}
\begin{aligned}
u_t(t,x,0) - \lim_{y \to 0} y^a \partial_y u(t,x,y) &= 0 && \text{on } (-1,0] \times B_1^0 \\
\dv y^a \grad u &= 0 && \text{in } [-1,0] \times B_1^+
\end{aligned}
\end{equation}

For any time $t\in (0,1)$, the solution of \eqref{e:fh-extended} is going to be $C^{\infty}$ in in the interior of $B_1^0$. Moreover, $u$ and all its partial derivatives of any order respect to $x$ are Lipschitz in time. The smoothness with respect to $y$ can then be understood by Lemma \ref{l:ext-flocal}.

\begin{prop} \label{p:fh-cinfinity}
Let $u$ be a bounded continuous solution of the equation \eqref{e:fh-extended}, then $u$ and $u_t$ are $C^{\infty}$ with respect to the $x$ variable in the interior of the domain. Moreover, for any multiindex $\alpha$, the following estimate holds
\begin{align*}
||\partial^\alpha_x u ||_{L^\infty([-1/2,0] \times B_{1/2}^+)} \leq C_\alpha \norm{u}_{L^\infty([-1,0] \times B_1^+)}, \\
||\partial^\alpha_x \partial_t u ||_{L^\infty([-1/2,0] \times B_{1/2}^0)} \leq C_\alpha \norm{u}_{L^\infty([-1,0] \times B_1^+)}.
\end{align*}
Here the constant $C_\alpha$ depends on $\alpha$, $s$ and $n$ only.
\end{prop}

\begin{proof}
Note that if we show an estimate for $u$ in $C^\infty_x(B_{3/4}^0)$, then the $C^\infty_x$ estimate in $B_{1/2}^+$ follows by the fact that, for each time, $u$ is a bounded solution of the Dirichlet problem for the elliptic equation of the extension. The same argument cannot be made for $u_t$ since a priori we cannot conclude that $u_t$ is bounded in the extension where $y>0$.

By a classical translation and scaling argument, it is enough to show the estimate at the point $(0,0,0)$.

Let $h(t,x)$ be the fractional heat kernel,
\begin{align*}
h_t + (-\lap)^s h &= 0 \qquad \text{in } (0,+\infty) \times \R^n\\ 
h(x,0) &= \delta_0
\end{align*}

There is no known closed form for $h(t,x)$ in real variables. However, in Fourier variables it is simply $\hat h(t,\xi) = e^{-t |\xi|^{2s}}$. Let us extend $h$ to the upper half space $h(t,x,y)$. Using the scale invariance of the equation, we have that for some function $H : \R^n \times \R^+ \to \R^+$,
\begin{equation} \label{e:heatkernel}
 h(t,x,y) = t^{-\frac n {2s}} H\left( t^{-\frac 1 {2s}} x , t^{-\frac 1 {2s}} y \right)
\end{equation}
The function $H$ is the extension to the upper half space to the function whose Fourier transform is $e^{-|\xi|^{2s}}$. It is bounded, smooth in $x$, and $H(x,0) \approx (1+|x|)^{-n-2s}$. From the Poisson formula in \cite{caffarelli2007extension} we also obtain $|D^k H(x,y)| \leq C (1+y)^{2s}(1+|x|+y)^{-n-2s-k}$. The function $h$ is bounded and $C^\infty$ in $x$ and $t$ except around the point $(0,0,0)$.

This function $h$ is the heat kernel of fractional diffusion in the whole space. We will multiply it by a smooth cutoff function in order to find an expression of $u(x,y,t)$ in terms of localized integral quantities. Let $\eta$ be a smooth radially symmetric function in $x$ and $y$, supported in $B_{7/8}^+$ and identically one in $B^+_{3/4}$. Note that since $\eta$ is radially symmetric and smooth, $\partial_y \eta(x,y) = O(|y|)$ as $y \to 0$. In particular $\lim_{y \to 0} y^a \partial_y \eta(x,y) = 0$.

At any point $x \in B_{1/2}^0$, we can recover $u(t,x)$ by the formula
\begin{equation} \label{e:endpointheat}
 u(x,t) = \lim_{\tau \to 0^+} \int_{B_1^0} u(\tau,z) h(t-\tau,x-z) \eta(x-z) \dd z.
\end{equation}


This formula allows us to compute estimates for solutions $u$ of \eqref{e:fh-extended} in a bounded domain.

Let $\tilde h(t,x,y) = h(t,x,y) \eta(x,y)$. Note that since $\partial_y \eta(x,y) = O(y)$ near $\{y=0\}$, $\tilde h$ satisfies
\[ \tilde h_t - \lim_{y \to 0} y^a \partial_y \tilde h = 0. \]
and also $\dv y^a \grad \tilde h$ is a bounded function supported in $\overline B_{7/8}^+ \setminus B_{3/4}^+$ which is $C^\infty$ in $x$. Let us analyze its behavior as $t \to 0$.
\begin{align*}
\dv y^a \grad \tilde h(t,x,y) &= \dv \left( y^a \grad \left( h \eta \right) \right) \\
&= h \dv \left( y^a \grad \eta \right) + 2 y^a \grad h \cdot \grad \eta \\
&= t^{-\frac n {2s}} H(t^{-\frac 1 {2s}} x,t^{-\frac 1 {2s}} y) \dv \left( y^a \grad \eta \right) + 2 y^a t^{-\frac n {2s} - 1} \grad H(t^{-\frac 1 {2s}} x,t^{-\frac 1 {2s}} y) \cdot \grad \eta \\
&\approx C(x,y).
\end{align*}
The last estimate uses that $H(x,y) \approx (|x|+y)^{-n-2s}$ and $\grad H(x,y) \approx (|x|+y)^{-n-2s-1}$ for large $|x|+y$.


Using \eqref{e:endpointheat}, we compute
\begin{align*}
u(0,x,0) &- \int_{B_1^0} u(-1,z) h(1,x-z) \dd z = \int_{-1}^0 \partial_t \int_{B_1^0} u(t,z) \tilde h(-t,x-z) \dd z \dd t \\
&= \int_{-1}^0 \int_{B_1^0} u_t(t,z) \ \tilde h(-t,x-z)  - u(t,z) \ \tilde h_t(-t,x-z) \dd z \dd t, \\
&= \int_{-1}^0 \int_{B_1^0} -(-\Delta)^s u(t,z) \ \tilde h(-t,x-z) \\ &\qquad \qquad \qquad + u(t,z) \ (-\Delta)^s h(-t,x-z) \ \eta(x-z) \dd z \dd t\\
&= \int_{-1}^0 \int_{B_1^+} u(t,z,y) \ \dv \left( y^a \grad \left( \tilde h(-t,x-z,y) \right) \right) \dd z \dd y\dd t\\
\end{align*}

Now we have $u(0,x,0)$ written as a sum of two convolutions with functions which are $C^\infty$ in $x$ and compactly supported. Therefore $u$ is $C^\infty$ in $x$ and the first estimate of the Lemma follows.

From the equation \eqref{e:fh-extended}, we have that
\[ u_t = \lim_{y \to 0} y^a \partial_y u.\]
Since $u$ is $C^\infty$ in $x$, the right hand side is well defined and $C^\infty$ in $x$. From this we deduce the second estimate of the Lemma.


\end{proof}

\begin{remark}
It is impossible to find a better regularity estimate in time in terms of the $L^\infty$ norm of $u$ only. The problem is that the estimate for $\partial_t u$ on $B^0_{1/2}$ cannot be extended to $B^+_{1/2}$. One can show that if $u=0$ on $(0,1/2) \times (\partial B_1)^+$ and $u=1$ on $(1/2,1) \times (\partial B_1)^+$, then indeed the solutions to \eqref{e:fh-extended} will be only Lipschitz in time on $\{1/2\} \times B_{1/2}^0$. However, just assuming that $u$ is continuous on $(0,1) \times (\partial B_1)^+$ would imply that $u$ is $C^1$ in time from the expression of $u_t = \lim_{y \to 0} y^a \partial_y u$ as a convolution.
\end{remark}

\section{Boundary behavior for the fractional heat equation}
\label{s:boundaryEstimatesForFHE}
In this section we study the H\"older continuity of the solution to the fractional heat equation \eqref{e:fh-extended} on the parabolic boundary. All the estimates in this section are proved by comparing the actual solution with explicit barriers. The results are very natural. They simply say that if the Dirichlet boundary data for the fractional heat equation is H\"older continuous, then the solution is H\"older continuous on the boundary. We tried to make the computations of the barriers as simple as possible, but some tedium is unavoidable. On a first reading of this paper, it would be advisable to skim through the results in this section.

The estimates in this section follow from the local properties of the fractional heat equation \eqref{e:fh-extended}. The equation is properly a parabolic evolution equation on $[-1,0] \times B_1^0$. For each fixed value of $t$, $u$ solves the elliptic equation $\dv y^a \grad u = 0$ on $\{t\} \times B_1^+$. We can expect to find regularity estimates on $[-1,0] \times B_1^0$ and also on $\{t\} \times B_1^+$ for each $t$. However, in this section we do not intend to analyze a modulus of continuity with respect to $t$ for $u(t,x,y)$ for positive values of $y$, since intuitively those quantities are not connected by the equation.

We start by constructing a barrier that will be useful to understand the continuity of $u$ close to $\partial B_1$ for each fixed $t$.

\begin{lemma} [Barrier for $(\partial B_1)^+$] \label{l:barrier1}
Let $\alpha \in (0,1)$ and $X_0=(x_0,y_0)$ be any point on $\partial B_1^+ \setminus \{y=0\}$. There exists a constant $C>0$ such that the function
\[ B(X) = C(1-|X|^2)^{\alpha} + |X-X_0|^\alpha \]
is a supersolution of
\[ \dv y^a \grad B \leq 0 \qquad \text{in } B_1^+. \]
\end{lemma}

\begin{proof}
One way to simplify the computations is to note that
\[ \dv y^a \grad B = y^a \left( \lap B + \frac a y \partial_y B \right). \]

Let us use this formula to estimate the equation in the two terms of the definition of $B$. We compute
\begin{align*}
\dv y^a \grad \left( (1-|X|^2)^{\alpha} \right) &= y^a \left( \lap (1-|X|^2)^{\alpha} + \frac a y \partial_y (1-|X|^2)^{\alpha} \right), \\
&= y^a \left( -\alpha 2 (n + a) (1-|X|^2)^{\alpha-1} +
\alpha(\alpha - 1) 4 |X|^2  (1-|X|^2)^{\alpha-2} \right).\\
&= y^a \alpha(1-|X|^2)^{\alpha-2} \left( -2 (n+a) (1-|X|^2) + 4 |X|^2 (\alpha -1) \right)
\end{align*}
This is a negative value since $\alpha<1$. Moreover, it becomes very negative as $|X|$ gets close to $1$.

Now we estimate the second term.
\begin{align*}
\dv y^a \grad \left( |X-X_0|^\alpha \right) &= \alpha y^a \left(n+\alpha-1 + a \frac{y-y_0}{y} \right) |X-X_0|^{\alpha-2} \\
\intertext{Since $y_0$ is positive and $a>0$,}
&\leq \alpha y^a \left(n+\alpha-1 + a \right) |X-X_0|^{\alpha-2}
\end{align*}

For all $X \in B_1^+$, $(1-|X|^2)^{\alpha-2} > c|X-X_0|^{\alpha-2}$, so we can choose a constant $C >0$ such that
\[ \dv y^a \grad \left( C (1-|X|)^{\alpha} + |X-X_0|^\alpha \right) \leq 0\]
for all $X \in B_1^+$.
\end{proof}

\begin{lemma} [Continuity on $(\partial B_1)^+$] \label{l:stationary-boundaryofball}
Let $u: \overline B_1^+ \to \R$ be the solution to
\begin{align*}
\dv y^a \grad u &= 0  \qquad \text{in }  B_1^+
\end{align*}
Assume that for some $(x_0,y_0) \in (\partial B_1)^+$ and every $(x,y) \in \partial B_1^+$ we have
\[ |u(x_0,y_0) - u(x,y)| \leq C_0 ((x-x_0)^2+(y-y_0)^2)^\alpha \]
Then for some $C>0$,
\[ |u(x_0,y_0) - u(x,y)| \leq C C_0 ((x-x_0)^2+(y-y_0)^2)^{\alpha} \]
for all $(x,y) \in B_1^+$. Here $C$ depends on $s$, $n$ and $\alpha$.
\end{lemma}

\begin{proof}
From the assumption, the function $U(x,y) = u(x_0,y_0) + C_0 B(x,y)$ is larger or equal than $u$ on $\partial B_1^+$, where $B(x,y) = B(X)$ is the supersolution constructed in Lemma \ref{l:barrier1}. Therefore, from comparison principle $u \leq U$ in the whole half ball $B_1^+$. Therefore
\[ u(x,y) - u(x_0,y_0) \leq C_0 B(x,y) \leq C_0 C (|x-x_0|+|y-y_0|)^{\alpha}. \]

\end{proof}

The next barrier is useful to analyze the H\"older continuity close to $B_1^0$ of solutions to $\dv y^a \grad u = 0$ in $B_1^+$. This barrier is not used in the rest of this section.
\begin{lemma} [Barrier for $B_1^0$] \label{l:barrier2}
Let $\alpha \in (0,1-a)$ and $x_0 \in B_1^0$. There is a $C \geq 1$ such that the function
\[ B(x,y) = |X-(x_0,0)|^\alpha + C y^{\alpha} \]
is a supersolution in the upper half space
\[ \dv y^a \grad B \leq 0 \qquad \text{for } y \in (0,1). \]
\end{lemma}

The proof is a direct computation similar to the one of Lemma \ref{l:barrier1}, so we omit it.





We now construct an auxiliary function that will be useful to construct barriers. Let $B$ be the only bounded solution of the following problem.
\begin{align*}
B(x,y) &= 0 \qquad \text{on } (\partial B_1)^+\\
-\lim_{y \to 0} y^a B(x,y) &= 1 \qquad \text{in } B_1^0 \\
\dv y^a \grad B &= 0  \qquad \text{in } B_1^+
\end{align*}

The explicit formula for $B$ can be computed using the intuition from \cite{caffarelli2007extension} that the equation in the upper half plane can be understood as the Laplace equation in fractional dimension.

\begin{equation} \label{e:formulaForB}
B(X) = \int_{B_1^0} \Phi(z-X) - |X|^{-n-a+1} \Phi(z-X/|X|^2) \dd z
\end{equation}
where $\Phi(X) = 1/|X|^{n-1+a}$ for any $X \in \R^{n+1}$ is the fundamental solution at the origin. It is elementary to check that this is the correct formula for $B$.

From the formula \eqref{e:formulaForB} we can see a couple of elementary properties of $B$:
\begin{enumerate}
\item There exists $c>0$ such that $B(x,0) \geq c (1-|x|)^s$.
\item The maximum of $B$ is achieved at $x=y=0$. Moreover there exists $c>0$ such that $B(0,0)-B(x,y) \geq c(|x|^2+y^{1-a})$.
\end{enumerate}

\begin{lemma} [Continuity on $\partial B_1^0 \times ( -1,0 )$] \label{l:fhe-regularitylateralboundary}
Let $u$ be a solution to the localized fractional heat equation \eqref{e:fh-extended}. Assume that for some $x_0 \in \partial B_1^0$ and $t_0 \in [-1,0]$ we know
\[ |u(t,x,y) - u(t_0,x_0,0)| \leq C_0 \left( |x-x_0|^2 + y^{1-a} + (t_0-t) \right)^\alpha \]
for all $(t,x,y) \in (\{-1\} \times B_1^+) \cup ([-1,t_0] \times (\partial B_1)^+)$. 
Then there is a constant $C>0$ such that
\[ |u(t,x,y) - u(t_0,x_0,0)| \leq C C_0 \left( |x-x_0|^2 + y^{1-a} + (t_0-t) \right)^{s\alpha/2} \ \ \text{for all } (t,x,y) \in [-1,t_0] \times \overline{B_1^+}. \]
Here $C$ depends on $\alpha$, $n$ and $a$ but not on $t_0$.
\end{lemma}

\begin{proof}
Let $r = (|x-x_0|^2/4 + y^{1-a}/(1-a) + (t_0-t))$. Let $U$ be the following barrier function
\[ U(t,x,y) = (x_0-x) \cdot x_0 +  \frac 1 {1-a} y^{1-a} + (t_0-t) + 2 B(x,y) \]
where $B(x,y)$ is the function in \eqref{e:formulaForB}.

We see that $r \leq U \leq Cr^{s/2}$ in $[-1,t_0] \times \overline B_1^+$.

For a given $(t_1,x_1,y_1) \in [-1,t_0] \times B_1^+$, let us choose $\rho = 4 r_1^{s/2}$ (where $r_1 = |x_1-x_0|^2/4 + y_1^{1-a}/(1-a) + (t_0-t_1)$) so that on $[-1,t_0] \times (\partial B_1)^+ \cup \{-1\} \times B_1^+$ we have
\[ C_0 \rho^{\alpha-1} U + C_0 \rho^\alpha \geq u(t,x,y) - u(t_0,x_0,0) \]
Since both the left hand side and right hand side solve the localized fractional heat equation, then also for all $(t,x,y) \in [-1,t_0] \times \overline B_1^+$
\[ u(t,x,y) - u(t_0,x_0,0) \leq C_0 \rho^{\alpha-1} U + C_0 \rho^\alpha \] 
In particular  \[ u(t_1,x_1,y_1) - u(t_0,x_0,0) \leq \leq C C_0 r_1^{s\alpha/2}. \]
The inequality from the other side follows by applying the same barrier bound to $u(t_0,x_0,0) - u$.
\end{proof}

\begin{lemma}[Continuity on $t=0$]\label{l:fhe-regularitybottomboundary}
Let $u$ be a solution to the localized fractional heat equation \eqref{e:fh-extended}.
Assume that for some $x_0 \in B_1^0$ we know
\begin{equation} \label{e:h-bbfhe}
|u(t,x,y) - u(-1,x_0,0)| \leq C_0 \left( |x-x_0|^2 + y^{1-a} + t+1 \right)^\alpha \ \ \end{equation}
for all $(t,x,y) \in \{-1\} \times B_1^+ \cup [-1,0] \times \partial B_1^+$.
Then there is a constant $C>0$ such that
\[ |u(t,x,y) - u(-1,x_0,0)| \leq C C_0 \left( |x-x_0|^2 + y^{1-a} + t+1 \right)^{s\alpha/2} \ \ \text{for all } (t,x,y) \in [0,T] \times \overline{B_1^+}. \]
where $C$ depends on $\alpha$, $n$ and $a$.
\end{lemma}

\begin{proof}
We will use the following barrier function
\[ U(t,x,y) = |x-x_0|^2 + \frac n {(1+a)} \left( -y^2 + 2 y^{1-a} \right) + 2n(t+1) \]
which is a solution to \eqref{e:fh-extended}. Let $r = |x-x_0|^2 + y^{1-a} + t+1$. Note that there are constants $c$ and $C$ such that $cr \leq U \leq Cr$.

From \eqref{e:h-bbfhe}, 
\[ |u(t,x,y) - u(-1,x_0,0)| \leq C_0 r^\alpha \ \ \text{for all } (t,x,y) \in \{-1\} \times B_1^+ \cup [-1,0] \times (\partial B_1)^+ \]
Therefore, for any $r_1>0$, on the parabolic boundary $(t,x,y) \in \{-1\} \times B_1^+ \cup [-1,0] \times (\partial B_1)^+$ we have
\[ u(t,x,y) - u(-1,x_0,0) \leq C_0 c r_1^{\alpha-1} U + C_0 r_1^\alpha \]
For any point $(t_1,x_1,y_1)$ and corresponding value of $r_1$, we apply comparison principle to obtain that
\[ u(t,x_1,y_1) - u(-1,x_0,0) \leq C_0 C r_1^\alpha. \]
Using the same argument to bound $u(-1,x_0,0) - u$ we finish the proof.
\end{proof}

\begin{lemma} [Main result of this section] \label{l:parabolic-boundary-continuity}
Let $u$ be a solution to the localized fractional heat equation \eqref{e:fh-extended}. Assume that $||u||_{L^\infty([-1,0] \times B_1^+)} \leq C_0$ and the following regularity for the boundary data:
\begin{align}
||u||_{C^\alpha(\{-1\} \times B_1^+)} &\leq C_0 \label{e:pbc1}\\
||u||_{C^\alpha([-1,0] \times \partial B^0_1)} &\leq C_0 \label{e:pbc2}\\
||u(t,.,.)||_{C^\alpha((\partial B_1)^+)} &\leq C_0 \ \ \text{for all } t \in [-1,0] \label{e:pbc3}
\end{align}
There is a constant $C$ such that if $(t,x,y)$ is any point in $[-1,0] \times B_1^+$ and $(t^*,x^*,y^*)$ is either its closest point on $\{t\} \times (\partial B_1)^+$ if $y>0$ or its closest point on $(\{-1\} \times B_1^0) \cup ([-1,0] \times \partial B_1^0)$ if $y=0$, then 
\begin{equation} \label{e:pbcc}
 |u(t,x,0) - u(t^*,x^*,y^*)| \leq C C_0 (|x-x^*|+|y-y^*|+|t-t^*|)^{\tilde \alpha}
\end{equation}
for some $\tilde \alpha > 0$ depending on $s$, $n$ and $\alpha$ only.
\end{lemma}

\begin{proof}
We will start the proof by analyzing the case $y=0$. In this case $y^*=y=0$.

The point $(t^*,x^*,0)$ belongs either to the bottom of the parabolic boundary: $\{-1\} \times B_1^0$ or the lateral boundary $[-1,0] \times \partial B_1^0$. We will analyze each of the two cases.

If $(t^*,x^*,0) \in \{-1\} \times B_1^0$, then from \eqref{e:pbc1}, \eqref{e:pbc2} and \eqref{e:pbc3}, the hypothesis of Lemma \ref{l:fhe-regularitybottomboundary} are satisfied (with $\alpha/2$ instead of $\alpha$). So we apply Lemma \ref{l:fhe-regularitybottomboundary} and prove \eqref{e:pbcc} with $\tilde \alpha = s^2 \alpha/2$

If $(t^*,x^*,0) \in [-1,0] \times \partial B_1^0$, then from \eqref{e:pbc1}, \eqref{e:pbc2} and \eqref{e:pbc3}, the hypothesis of Lemma \ref{l:fhe-regularitylateralboundary} are satisfied (with $\alpha/2$ instead of $\alpha$). So we apply Lemma \ref{l:fhe-regularitylateralboundary} and prove \eqref{e:pbcc} with $\tilde \alpha = s^2 \alpha/2$

So we have finished the proof in the case $y=0$. What we proved so far implies that for every $(t_0,x_0,y_0) \in [-1,0] \times (\partial B_1)^+$ and $(t_0,x,y) \in [-1,0] \times \partial B_1^+$ (recall $\partial B_1^+ = (\partial B_1)^+ \cup B_1^0$), we have
\begin{equation} \label{e:pbc4}
|u(t^*,x,y) - u(t^*,x^*,y^*)| \leq C C_0 (|x-x^*|+|y-y^*|)^{\tilde \alpha} \end{equation}

Let us analyze the case $y > 0$ now. In this case $(x^*,y^*) \in \partial B_1$. On the other hand, from \eqref{e:pbc3} and \eqref{e:pbc4}, we can apply Lemma \ref{l:stationary-boundaryofball} with $x_0 = x^*$ and $y_0=y^*$ and finish the proof.
\end{proof}

\begin{remark}
Note that the barriers in this section could be used (together with the comparison principle) to show the existence of the solution of the Dirichlet problem for \eqref{e:fh-extended} by Perron's method.
\end{remark}

\section{H\"older estimates for the equation with drift}
\label{s:holderEstimates}
Our assumption $b \in C^{1-2s+\alpha}$ means that the vector field $b$ is slightly better than $C^{1-2s}$. From \cite{silvestrePreprintHolder}, just from $b$ being in the class $C^{1-2s}$, we can obtain that the solution $u$ is H\"older continuous. In this section, we adapt the result of \cite{silvestrePreprintHolder} to the extended problem which allows us to localize the estimate.

\begin{prop} \label{p:holder-local-extended}
Let $u$ be a bounded function in $[-1,0] \times B_1^+$ solving
\begin{equation} \label{e:mainequation-extended}
\begin{aligned}
u_t +b \cdot \grad u- \lim_{y \to 0} y^a \partial_y u(t,x,y) &= f && \text{on }  (-1,0] \times B_1^0 \\
\dv y^a \grad u &= 0 && \text{in } [-1,0] \times B_1^+
\end{aligned}
\end{equation}
where $b$ is a vector field in $[-1,0] \times B_1$ which is $C^{1-2s}$ in $x$ and $f$ is a scalar function in $L^\infty([-1,0]\times B_1^0)$. Then $u$ is H\"older continuous in $[-1/2,0] \times B_{1/2}^0$ and also in $\{t\} \times B_{1/2}^+$ for every $t \in [-1/2,0]$
\[ ||u||_{C^\alpha( [-1/2,0] \times B_{1/2}^0)} + ||u||_{L^\infty([-1/2,0],C^\alpha(B_{1/2}^+))} \leq C ||u||_{L^\infty((-1,0] \times B_1^+)} \]
for some constants $\alpha>0$ and $C$ depending on $s$, $\norm{b}_{C^{1-2s}}$ and the dimension $n$.
\end{prop}
  
\begin{proof}
Consider the function $v : [-1,0] \times \R^n \times \R^+$ defined by
\begin{align*}
v(t,x,0) &= \begin{cases}
u(t,x,0) & \text{if } |x| < 1 \\
0 & \text{otherwise}
\end{cases} \\
\dv y^a \grad v &= 0 \ \ \text{for } y>0
\end{align*}

Note that the function $u-v$ satisfies
\begin{align*}
(v-u)(t,x,0) &= 0 \ \ \text{in } B_1^0,\\
\dv y^a \grad (v-u) &=0 \ \ \text{in } B_1^+.
\end{align*}

Moreover, $\norm{v-u}_{L^\infty} \leq 2\norm{u}_{L^\infty}$.
Thus, for every fixed value of $t$ we can apply Lemma \ref{l:ext-flocal} to obtain that 
\[ g(t,x) = \lim_{y \to 0} y^a \partial_y (v-u)(t,x,y) \]
is differentiable in $x$ (with an estimate depending only on $\norm{u}_{L^\infty}$. In particular, $g$ is bounded on $[-1,0] \times B_{3/4}^0$.

Now, let us analyze the equation for $v$. We see that $v_t=u_t$ and $\grad v = \grad u$ on $[-1,0] \times B_1^0$. Therefore $v$ satisfies the equation.
\[
\begin{aligned}
v_y +b \cdot \grad v- \lim_{y \to 0} y^a \partial_y v(t,x,y) &= f+g && \text{on } (-1,0] \times B_1^0 \\
\dv y^a \grad v &= 0 && \text{in } [-1,0] \times B_1^+
\end{aligned}
\]
Since the right hand side $g$ is bounded in $[-1,0] \times B_{3/4}^0$, we can apply the main theorem from \cite{silvestrePreprintHolder} to obtain that $v$ is H\"older continuous in $[-1/2,0] \times B_{2/3}^+$. That means that $u$ is H\"older continuous on $[-1/2,0] \times B_{2/3}^0$. From this, one concludes that $u$ is H\"older continuous on $\{t\} \times B_{1/2}^+$ for each $t \in [-1/2,0]$ using the barrier from Lemma \ref{l:barrier2} and that $u(t,.,.)$ solves the equation $\dv y^a \grad u = 0$ in $B_1^+$.
\end{proof}

\section{The approximation lemma}
\label{s:perturbation}

The main lemma in this section says that if the vector field $b$ and the right hand side $f$ are very small, then $u$ differs very little from the solution to the fractional heat equation with the same boundary condition. This lemma is important in order to carry out our proof by perturbation. Indeed, the main idea of the proof of Theorem \ref{t:main} will be to approximate the solution $u$ with the smooth solution of the drift-less problem with zero right hand side.

\begin{lemma} \label{l:perturbation}
Let $u$ be a continuous function that solves 
\begin{equation} \label{e:perturbed-extended}
\begin{aligned}
u_t + b \cdot \grad u- \lim_{y \to 0} y^a \partial_y u(t,x,y) &= f && \text{on } (-1,0] \times B_1^0 \\
\dv y^a \grad u &= 0 && \text{in } [-1,0] \times B_1^+
\end{aligned}
\end{equation}
Assume that $\norm{u}_{L^\infty} \leq 1$, $\norm{u}_{C^\alpha([-1,0] \times B_1^0)} \leq C$ and $\norm{u}_{L^\infty([-1,0],C^\alpha(B_1^+))} \leq C$ for some $\alpha>0$. Let $v$ be a solution of the driftless problem
\begin{equation} \label{e:perturbed-driflessproblem}
\begin{aligned}
&v_t(t,x,0) - \lim_{y \to 0} y^a \partial_y v(t,x,y) = 0 \ \ \text{on } (-1,0] \times B_1^0 \\
&\dv y^a \grad v = 0 \ \ \text{in }  [-1,0] \times B_1^+ \\
&v(t,x,y) = u(t,x,y) \ \  \text{on the parabolic boundary: } ([-1,0] \times (\partial B_1)^+) \cup (\{-1\} \times B_1^+).
\end{aligned}
\end{equation}
Then for every $\eps>0$ there exists a $\delta>0$ so that $\norm{u-v}_{L^\infty([-1,0] \times B_1^+)} < \eps$ every time that $||b||_{L^\infty([-1,0] \times B_1^0)} < \delta$ and $||f||_{L^\infty([-1,0] \times B_1^0)} < \delta$.
\end{lemma}

\begin{proof}
From Lemma \ref{l:parabolic-boundary-continuity}, the for any $(x,y,t) \in [-1,0] \times B_1^+$ and $(x^*,y^*,t^*)$ as in Lemma \ref{l:parabolic-boundary-continuity}, we have
\[ |v(x,y,t) - v(x^*,y^*,t^*)| \leq C(|x-x^*|+|y-y^*|+|t-t^*|)^{\tilde \alpha}\]
On the other hand, since $\norm{u}_{C^\alpha([-1,0] \times B_1^0)} \leq C$ and $\norm{u}_{L^\infty([-1,0],C^\alpha(B_1^+))} \leq C$, we also have
\[ |u(x,y,t) - u(x^*,y^*,t^*)| \leq C(|x-x^*|+|y-y^*|+|t-t^*|)^{\tilde \alpha}\]
But $v(x^*,y^*,t^*) = u(x^*,y^*,t^*)$. Therefore, if $(x,y,t)$ belongs to either $[-1,0] \times (B_1^+ \setminus B_{1-\rho}^+)$ or $[-1,-1+\delta] \times B_1^0$, we have
\[ |v(x,y,t) - u(x,y,t)| \leq C \rho^{\tilde \alpha}.\]
Moreover, on each time slice $t \in [-1,-1+\delta]$, we use maximum principle for the degenerate elliptic equation
\[ \dv y^a \grad (u-v) = 0 \ \ \text{on } \{t\} \times B_1^+.\]
Then $|u-v| \leq C \rho^{\tilde \alpha}$ in $[-1,-1+\delta] \times B_1^+$.
Choosing $\rho$ small enough, we can assure that $|u-v| < \eps/2$ in $([-1,0] \times B_1^+) \setminus ([-1+\rho,0] \times B_{1 - \rho}^+)$.

From Proposition \ref{p:fh-cinfinity}, $\grad v$ is bounded in $[-1+\rho,0] \times B^0_{1-\rho}$ with a bound depending only on dimension, $a$, $\rho$ and $\norm{v}_{L^\infty} = 1$. Therefore
\[ v_t(t,x,0) + b \cdot \grad v- \lim_{y \to 0} y^a \partial_y v(t,x,y) \leq C \delta \ \ \text{ in } [-1+\rho,0] \times B^0_{1-\rho}\]

We can conclude using maximum principle that $\tilde v = v - \eps/2 - C \delta t \leq u$. Indeed, we know that $\tilde v \leq u$ on the boundary and complement of $[-1+\rho,0] \times B_{1-\rho}$, whereas inside we have
\[ \tilde v_t(t,x,0) + b \cdot \grad \tilde v- \lim_{y \to 0} y^a \partial_y \tilde v(t,x,y) \leq -\delta. \]
Similarly, we prove that $v + \eps/2 + C \delta t \geq u$. We finish the proof by choosing $\delta$ small so that $C \delta < \eps/2$.
\end{proof}

\begin{remark} \label{r:perturbation0}
Note that the fact that $u$ solves the equation \eqref{e:dd-extended} is used only to apply the comparison principle with the classical subsolution $\bar v$, thus $u$ only needs to satisfy \eqref{e:dd-extended} in the viscosity sense. 
\end{remark}

\begin{remark} \label{r:perturbation}
The proof of Lemma \ref{l:perturbation} adapts with little work to the vanishing viscosity framework. Indeed, if the equation for $u$ is
\begin{equation} \label{e:perturbed-viscosity}
\begin{aligned}
u_t + b \cdot \grad u- (\lim_{y \to 0} y^a \partial_y u) - \eps_0 \lap u &= f && \text{on } (-1,0] \times B_1^0 \\
\dv y^a \grad u &= 0 && \text{in } [-1,0] \times B_1^+
\end{aligned}
\end{equation}
then under the same hypothesis plus $\eps_0 < \delta$, the solution to \eqref{e:perturbed-driflessproblem} approximates \eqref{e:perturbed-viscosity}. The only modification in the proof is that when we analyze the equation for $v$ we also have the term $\eps_0 \lap v$, which is small since by Proposition \ref{p:fh-cinfinity}, $v$ is $C^2$ in $x$.
\end{remark}

\section{The improvement of flatness lemma}
\label{s:improvementOfFlatness}
Our main regularity result is proved iteratively by showing that the values of the solution separate less and less from a plane in smaller scales. In each iteration we use the essential \emph{improvement of flatness} lemma which is stated and proved below.

\begin{lemma} \label{l:improvement-of-flatness}
Let $\alpha$ be any positive number less than $2s$, $b \in C^{1-2s}$, and $u$ be a continuous function that solves 
\begin{equation} \label{e:p-extended}
\begin{aligned}
u_t + b \cdot \grad u- \lim_{y \to 0} y^a \partial_y u(t,x,y) &= f && \text{on } (-1,0] \times B_1^0 \\
\dv y^a \grad u &= 0 && \text{in } [-1,0] \times B_1^+
\end{aligned}
\end{equation}
Assume that $\norm{u}_{L^\infty} \leq 1$ and $\norm{b}_{L^\infty} \leq \delta$ and $||f||_{L^\infty} < \delta$. Then, if $\delta$ is small enough, there is a radius $r>0$, a vector $A \in \R^n$ and a $C^1$ function $D:[-r^{2s},0] \to \R$ such that
\begin{equation} \label{e:flatter}
\sup_{[-r^{2s},0] \times B_r^+} |u(t,x,y) - A \cdot x - D(t) - \frac {D'(t)} {1-a} y^{1-a}| \leq r^{1+\alpha}
\end{equation}
Moreover, $|A|$ is bounded by a universal constant $C$ depending only on dimension, $s$, and $||b||_{C^{1-2s}}$ and so are $|D(t)|$ and $|D'(t)|$.
\end{lemma}

The idea of the proof is to use Lemma \ref{l:perturbation} from the previous section to show that $u$ differs very little from a solution $v$ to the drift-less problem with zero right hand side. This function $v$ is smooth, so $u$ is close to a smooth function and the result follows.

\begin{proof}[Proof of Lemma \ref{l:improvement-of-flatness}]
From Proposition \ref{p:holder-local-extended}, the function $u$ is H\"older continuous in $[-1/2,0] \times B_{1/2}^0$ and on $\{t\} \times \overline{B_{1/2}^+}$ for all $t \in [-1/2,0]$. 

Let $v$ be the solution to the fractional heat equation in $[-1/2,0] \times B_{1/2}^+$ with $u$ as boundary values:
\begin{equation} \label{e:p-driflessproblem}
\begin{aligned}
v_t(t,x,0) - \lim_{y \to 0} y^a \partial_y v(t,x,y) &= 0 && \text{on } (-1/2,0] \times B_{1/2}^0\\
\dv y^a \grad v &= 0 && \text{in } [-1/2,0] \times B_{1/2}^+ \\
v(t,x,y) &= u(t,x,y) && \text{on } ([-1/2,0] \times \partial B_1^+) \cup (\{-1/2\} \times B_{1/2}^+).
\end{aligned}
\end{equation}

From Lemma \ref{l:perturbation}, $|u-v|\leq \eps$ in $[-1/4,0] \times B_{1/4}^+$, where $\eps$ is arbitrarily small depending on the choice of $\delta$.

From Proposition \ref{p:fh-cinfinity}, the function $v$ is $C^\infty$ respect to $x$, and $\partial_t v$ and $\partial_t \grad v$ are bounded in the interior of $[-1/2,0] \times B_{1/2}^+$ with estimates depending on the oscillation of $v$, which is less than one. In particular, if $A = \grad v(0,0,0)$ and $D = v(t,0,0)$, $v$ satisfies
\[ |v(t,x,0) - D(t) - A \cdot x| \leq C(|x||t| + |x|^2) \leq C r^{1+2s} \qquad \text{in } [-r^{2s},0] \times B_r^+. \]

Also from Proposition \ref{p:fh-cinfinity}, $D'(t)$ is bounded depending only on dimension and $s$.

Since $v$ satisfies the equation \eqref{e:p-driflessproblem}, then $D'(t) = \partial_t v(t,0,0) = \lim_{y \to 0} y^a \partial_y v(t,0,y)$. From Lemma \ref{l:ext-flocal}, $v$ has an expansion of the following form in $[-r^{2s},0] \times B_r^+$.
\[ \begin{split}
v(t,x,y) &= v(t,x,0) + \frac{y^{1-a}}{1-a} \lim_{y \to 0} y^a \partial_y v(t,x,y) + O(y^2) \\
&= A \cdot x + D(t) + \frac{D'(t)}{1-a} y^{1-a} + O(r^{1+2s})
\end{split}
\]
Meaning that the last error term $O(r^{1+2s})$ is bounded by $C r^{1+2s}=Cr^{2-a}$ for some constant $C$ independent of $u$ and $v$.

Therefore, if $\alpha < 2s$, we can choose $r$ small such that 
\[ \begin{split}
|v(t,x,y) -  D(t) - A \cdot x - \frac {D'(t)} {1-a} y^{1-a} | &\leq C r^{1+2s} \\
&\leq \frac 1 2 r^{1+\alpha}
\end{split}\]

Now we choose $\delta$ so that the $\eps$ that we obtained from Lemma \ref{l:perturbation} is less than $\frac 1 4 r^{1+\alpha}$. So we obtain
\[ |u(t,x,y) -  D(t) - A \cdot x - \frac {D'(t)} {1-a} y^{1-a}| \leq \frac 12 r^{1+\alpha} + 2 \eps \leq r^{1+\alpha}
\]
which finishes the proof.
\end{proof}

\section{The proof of the main theorem}
\label{s:mainProof}
We start the proof of our main theorem by proving the result in the case $b(t,0)=0$ and $f(t,0)=0$. The proof of the general case follows below and consist of combining this lemma with a change of variables following the flow of the vector field.

\begin{lemma} \label{l:main}
Let $\alpha \in (0,2s)$, and $u$ be a solution of
\begin{equation} \label{e:extended}
\begin{aligned}
u_t + b \cdot \grad u- \lim_{y \to 0} y^a \partial_y u(t,x,y) &= f && \text{on } (-1,0] \times B_1^0 \\
\dv y^a \grad u &= 0 && \text{in } [-1,0] \times B_1^+
\end{aligned}
\end{equation}
Assume $u(0,0,0)=0$, $||u||_{L^\infty([-1,0] \times B_1^+)} \leq 1$ and
\begin{equation} \label{e:b-small-all-scales}
\sup_{(t,x) \in [-1,0] \times B_1} \left( \frac{|b(t,x)|}{|x|^{1-2s+\alpha}} \right) \leq \delta \ \ \text{ and } \ \sup_{(t,x) \in [-1,0] \times B_1} \left( \frac{|f(t,x)|}{|x|^{1-2s+\alpha}} \right) \leq \delta
\end{equation}
(in particular $b(t,0)=0$ and $f(t,0)=0$ for all $t$). Then there exist $A \in \R^n$ and $D:[-1/2,0] \to \R$ such that
\begin{equation} \label{e:estimate-at-zero}
\abs{u(t,x,y) -  A \cdot x - D(t) - \frac{D'(t)}{1-a} y^{1-a}} \leq C (|x|+y+t^{\frac 1{2s}})^{1+\alpha}
\end{equation}
\end{lemma}

\begin{proof}
Let $\rho>0$ be the radius $r$ of Lemma \ref{l:improvement-of-flatness}. We will prove iteratively that there exists $A_k$ and $D_k$ uniformly bounded such that
\begin{equation} \label{e:iterative-inequality}
\sup_{(t,x,y) \in [-\rho^{2sk},0]\times B_{\rho^k}^+} \abs{u(t,x,y) -  A_k \cdot x - D_k(t) - \frac {D_k'(t)} {1-a} y^{1-a}} \leq \rho^{k(1+\alpha)}.
\end{equation}

For $k=0$, we choose $A=0$ and $D=0$ and \eqref{e:iterative-inequality} holds. Now, let us assume we have shown \eqref{e:iterative-inequality} up to some value of $k$ and prove it for $k+1$.

Let \[v(t,x,y) = \rho^{-k(1+\alpha)} \left(u(\rho^{2sk} t, \rho^k x, \rho^k y) - \rho^k \, A_k \cdot x - D(\rho^{2sk} t) - \rho^{-2sk} \frac {D'(\rho^{2sk} t)} {1-a} y^{1-a}\right)\] and $\tilde b = \rho^{(2s-1)k} b(\rho^{2sk} t, \rho^k x)$ and $\tilde f=\rho^{k(2s-1-\alpha)} f(\rho^{2sk}t,\rho^k x)$. We have that $v$ solves
\begin{equation} \label{e:scaled}
\begin{aligned}
v_t(t,x,0) + \tilde b \cdot \grad v - \lim_{y \to 0} y^a \partial_y v(t,x,y) &= \tilde f - \rho^{-k \alpha} \tilde b \cdot A_k&& \text{on } (-1,0] \times B_1^0, \\
\dv y^a \grad v &= 0 && \text{in } [-1,0] \times B_1^+.
\end{aligned}
\end{equation}
Note that $\tilde f - \rho^{-k \alpha} \tilde b \cdot A_k$ is small because of the assumption \eqref{e:b-small-all-scales} and the fact that $A_k$ is uniformly bounded (independently of $\delta$).

Note that the term $- D(t)$ adds a $-D'(t)$ to the derivative of $v$ respect to time, which is compensated by the term $- \frac {D'(t)} {1-a} y^{1-a}$ which does the same thing with the boundary derivative $\lim_{y \to 0} y^a \partial_y v(t,x,y)$.

From the inductive hypothesis, we know that $|v| \leq 1$ in $[-1,0] \times B_1^+$ and from \eqref{e:b-small-all-scales}, we have that $|\tilde b| \leq \delta$  and $|\tilde f|<\delta$ in $[-1,0] \times B_1^+$. We can apply Lemma \ref{l:improvement-of-flatness} to $v$ to obtain that there are $A$ and $D$ uniformly bounded such that
\[ \abs{v(t,x,y) - v(0,0,0) - A \cdot x - D(t) - \frac {D'(t)} {1-a} y^{1-a}} \leq \rho^{1+a}. \]
Therefore \eqref{e:iterative-inequality} holds for $k+1$ with $A_{k+1} = A_k + \rho^{k\alpha} A$ and $D_{k+1} = D_k + \rho^{k(\alpha+a)} D$. This finishes the proof of \eqref{e:iterative-inequality} by induction in $k$.

Note that by construction
\begin{align*}
|A_k| &\leq C \sum_{i<k} \rho^{\alpha i} \\
|D_k|, |D_k'|  &\leq C \sum_{i<k} \rho^{(\alpha+a) i}
\end{align*}
So $|A_k|$ and $|D_k|$ are bounded by a convergent geometric series, and then they stay bounded uniformly in $k$. Moreover, if we set
\begin{align*}
A &= \lim_{k \to \infty} A_k \\
D &= \lim_{k \to \infty} D_k
\end{align*}
then, also by the geometric series bound, $|A_k - A| \leq C \rho^{-k\alpha}$ and $|D_k - D| \leq C \rho^{-k(\alpha+a)}$. Therefore, \eqref{e:iterative-inequality} implies \eqref{e:estimate-at-zero} and we finish the proof. 
\end{proof}

\begin{remark}
In the proof of Lemma \ref{l:main}, we see the importance of rewriting the equation \eqref{e:dd} using the extension as in \eqref{e:dd-extended}. The slope with respect to time $Dt$ is compensated by subtracting $Dy^{1-a}/(1-a)$ in the extended (artificial) variable. The terms fit together nicely and the error terms from the nonlocal operators are completely avoided. Indeed, in this supercritical case, those error terms would be very difficult to control without using the extension.
\end{remark}

\begin{remark}
If $\alpha=0$, Lemma \ref{l:main} still applies with $A=0$, and all $A_k$ in the proof equal to zero as well.
\end{remark}

We now have all the results needed to prove the main theorem.

\begin{proof} [Proof of Theorem \ref{t:main}]
Let $u$ be a bounded solution to \eqref{e:dd} in $[-1,0] \times \R^n$. By the usual translation and dilation argument, it is enough to show that the estimate holds at the point $(0,0)$. We will prove that there is a vector $A \in \R^n$ such that
\begin{equation} \label{e:m1}
|u(0,x) - u(0,0) - A \cdot x| \leq C |x|^{1+\alpha} \left( ||u||_{L^\infty} + \frac{r^{1+\alpha}} \delta ||f||_{C^{1-2s+\alpha}} \right)
\end{equation}
where $\delta$ is the constant from Lemma \ref{l:main} and $r$ depends on $\norm{b}_{C^{1+2s-\alpha}}$.

We first normalize the solution by considering
\[ \begin{aligned}
\tilde u &= \frac {u(r^{2s} t,rx) - u(0,0)} {||u||_{L^\infty} + \frac{r^{1+\alpha}} \delta ||f||_{C^{1-2s+\alpha}}}, \\ 
\tilde f &= \frac {r^{2s} f(r^{2s} t,rx)} {||u||_{L^\infty}+\frac{r^{1+\alpha}} \delta ||f||_{C^{1-2s+\alpha}}} ,\\
\tilde b &= \frac {r^{2s-1} b(r^{2s} t,rx)} {||u||_{L^\infty}+\frac{r^{1+\alpha}} \delta ||f||_{C^{1-2s+\alpha}}}.
\end{aligned} \]
In this way, $||\tilde u||_{L^\infty} \leq 1$, and also $||\tilde f||_{C^{1-2s+\alpha}} \leq \delta$ and $||\tilde b||_{C^{1-2s+\alpha}} \leq \delta$ if $r$ is small.

Note that $\tilde u$, $\tilde b$ and $\tilde f$ also satisfy the equation \eqref{e:dd}. We would finish the proof if we could apply Lemma \ref{l:main} to $\tilde u$, $\tilde b$ and $\tilde f$. We extend $\tilde u$ and $\tilde f$ to the upper half space. Note that the extension of $u$ has the same $L^\infty$ norm as $u$ by the maximum principle (indeed, all that is needed is that the extension of $u$ is bounded in $[-1,0] \times B_1^+$). The only hypothesis we are missing in order to apply Lemma \ref{l:main} are \eqref{e:b-small-all-scales}. We will make another change of variables.

We solve the following ODE backwards in time:
\begin{align*}
V(0) &= 0, \\
\dot V(t) &= b(t,V(t))  \ \ \text{for } t \in (-1,0).
\end{align*}

Since $b$ is continuous, the ODE has at least one solution. Moreover, $|\dot V| \leq C$ since $b$ is bounded.

We also define \[ S(t) = \int_t^0 \tilde f(s,0) \dd s, \] which is clearly the solution of
\begin{align*}
S(0) &= 0 \\
\dot S(t) &= -f(t,0) 
\end{align*}

Therefore we now consider
\begin{align*}
 u^*(t,x,y) &= \tilde u(t,x+V(t),y) - S(t) \\
 b^*(t,x) &= b(t,x+V(t)) - b(t,V(t)) \\
f^*(t,x) &= \tilde f(t,x) - f(t,0).
\end{align*}
And they solve the equation
\begin{align*}
u^*_t + b^* \cdot \grad u^* - \lim_{y \to 0} y^a \partial_y u^* &= f^* \ \ \text{in } (-1,0] \times B_1^0 \\
\dv y^a \grad u^* &= 0 \ \ \text{in } [-1,0] \times B_1^+
\end{align*}
Moreover, since $||\tilde f||_{C^{1-2s+\alpha}} \leq \delta$ and $||\tilde b||_{C^{1-2s+\alpha}} \leq \delta$, then $f^*$ and $b^*$ satisfy \eqref{e:b-small-all-scales}. So, we can apply Lemma \ref{l:main} to $u^*$ to obtain that there is $A_* \in \R^n$ and $D_*:[-1/2,0] \to \R$ such that
\[ \abs{u^*(t,x,y) - u^*(0,0,0) -  A_* \cdot x - D_*(t) -  \frac {D_*'(t)} {1-a} y^{1-a}} \leq C (|x|+y+t^{\frac 1{2s}})^{1+\alpha}. \]
But note that $\tilde u(0,x,y) = u^*(0,x,y)$ and \[ u(0,x,y) = \left( ||u||_{L^\infty} + \frac{r^{1+\alpha}} \delta ||f||_{C^{1-2s+\alpha}} \right) v^*(0,r^{-1} x, r^{-1} y) \]
Therefore, if we set $A = A^* \left( ||u||_{L^\infty} + \frac{r^{1+\alpha}} \delta ||f||_{C^{1-2s+\alpha}} \right) r^{-1}$, we obtain
\[ \abs{u(0,x,0) - u^*(0,0,0) -  A \cdot x} \leq C_r |x|^{1+\alpha} \left( ||u||_{L^\infty} + ||f||_{C^{1-2s+\alpha}} \right), \]
which finishes the proof (recall that $r$ depends on $||b||_{C^{1-2s+\alpha}}$).
\end{proof}

\begin{proof}[Proof of Theorem \ref{t:main2}]
Repeating the proof of Theorem \ref{t:main} but with $\alpha=0$ and $A=0$, Theorem \ref{t:main2} follows. 
The smallness of $\tilde b$ follows from \eqref{e:mainassumption2} for small $r$.
\end{proof}


\begin{remark}
If we follow the iteration in the proof of Theorem \ref{t:main} in the vanishing viscosity framework, we see that the artificial viscosity term $\eps \lap u$ gets magnified in each successive scaling by $\rho^{k(s-1)}$. For a large value of $k$, this value would be above the threshold to apply Lemma \ref{l:perturbation}. But at that point the second order diffusion takes over, and the rest of the iteration is trivial using the regularization of the perturbed second order heat equation.
\end{remark}

\bibliographystyle{plain}   
\bibliography{pafd}             
\index{Bibliography@\emph{Bibliography}}%

\end{document}